\documentclass[12pt]{article}
\usepackage{amssymb}
\usepackage{amsmath}
\usepackage{amsthm}
\usepackage[latin1]{inputenc} 
\usepackage[T1]{fontenc}      
\usepackage{graphicx}
\usepackage{eurosym}
\usepackage{hyperref}
\date{October 16, 2012}
\newtheorem{theorem}{Theorem}

\newtheorem{lemma}[theorem]{Lemma}
\newtheorem{proposition}[theorem]{Proposition}
\newtheorem{remark}[theorem]{Remark}
\newtheorem{assumption}[theorem]{Assumption}

\newtheorem{example}[theorem]{Example}

\newcommand{\xad}{x_\alpha^\delta}

\newcommand{\xdag}{x^\dagger}

\newcommand{\range}{\mathcal{R}}
\newcommand{\nullspace}{\mathcal{N}}

\newcommand{\R}{\mathbb{R}}
\newcommand{\N}{\mathbb{N}}
\newcommand{\1}{\ell^1(\N)}
\newcommand{\2}{\ell^2(\N)}
\newcommand{\3}{\ell^\infty(\N)}
\newcommand{\M}{\mathcal{M}}

\title{\bf Convergence rates in $\mathbf{\ell^1}$-regularization if the sparsity assumption fails}

\author{{\sc Martin Burger}\thanks{Westf\"alische Wilhelms-Universit\"at M\"unster, Institut f\"ur Numerische und Angewandte Mathematik, Einsteinstr.~62, 48149 M\"unster, Germany, Email:$\;$\texttt{martin.burger\,@\,wwu.de}} , {\sc Jens Flemming}\thanks{Technische Universit\"at Chemnitz, Fakult\"at f\"ur Mathematik,  09107 Chemnitz,
Germany. Email:$\;\;$\texttt{jens.flemming\,@\,mathematik.tu-chemnitz.de,\, hofmannb\,@\,mathematik.tu-chemnitz.de}\,.} , and \,{\sc Bernd Hofmann}\footnotemark[2]}

\begin{document}

\maketitle

\begin{abstract}
Variational sparsity regularization based on $\ell^1$-norms and other nonlinear functionals has gained enormous attention recently, both with respect to its applications and its mathematical analysis. A focus in regularization theory has been to develop error estimation in terms of regularization parameter and noise strength. For this sake specific error measures such as Bregman distances and specific conditions on the solution such as source conditions or variational inequalities have been developed and used.

In this paper we provide, for a certain class of ill-posed linear operator equations, a convergence analysis that works for solutions that are not completely sparse, but have a fast decaying nonzero part. This case is not covered by standard source conditions, but surprisingly can be treated with an appropriate variational inequality. As a consequence the paper also provides the first examples where the variational inequality approach, which was often believed to be equivalent to appropriate source conditions, can indeed go farther than the latter.
\end{abstract}

\vspace{0.3cm}

{\parindent0em {\bf MSC2010 subject classification:}
65J20,} 47A52, 49N45

\vspace{0.3cm}

{\parindent0em {\bf Keywords:}
Linear ill-posed problems}, sparsity constraints, Tikhonov-type regularization, $\ell^1$-regularization, convergence rates, solution decay, variational inequalities, source conditions, distance function, regularization parameter choice.

\section{Introduction}\label{s1}
\setcounter{equation}{0}
\setcounter{theorem}{0}

Variational problems of the form
\begin{equation}
	 \frac{1}2 \Vert A x - y^\delta \Vert^2 + \alpha \Vert x \Vert_{\1}\to\min_{x\in\1}
\end{equation}
have become an important and standard tool in the regularization of operator equations $Ax=y$. In the field of compressed sensing, with a usually not injective but quite well-conditioned finite-dimensional operator $A$ the sparsity modeled via the $\ell^1$-minimization yields the appropriate prior knowledge to uniquely restore the desired solution. In inverse problems, with usually an infinite-dimensional compact operator $A$, the sparsity prior allows for stable reconstructions even in presence of noise. There is a comprehensive literature concerning the
$\ell^1$-regularization of ill-posed problems under sparsity constraints including assertions on convergence rates (cf. e.g. \cite{Daub03} and  \cite{BeBur,BurResHe,BreLor09,Gras11,GrasHaltSch08,GrasHaltSch11,Lorenz08,moeller,RamRes10,RamTe10,Ramlau08,Scherzetal09}).

A natural question arising in problems of this type is the asymptotic analysis of such variational problems as $\alpha \rightarrow 0$ respectively $y^\delta \rightarrow y$, where $y$ are the data that would be produced by an exact solution of the problem. While it is straight-forward to show convergence of subsequences in weak topologies, quantitative error estimates are much more involved. For a long time it was even an open question what are the right conditions and error measures to perform such an analysis. In the last years, following \cite{BurOsh04}, the use of Bregman distances has been evolved as a standard tool. First estimates were based on source conditions of the form (cf. \cite{BurResHe} for the $\ell^1$-case and
\cite{moeller} for further analysis in a compressed sensing framework)
\begin{equation}
	\exists w: \quad A^* w \in \partial \Vert x^\dag \Vert_{\1},
\end{equation}
where $x^\dag$ is the exact solution. Extensions to approximate source conditions (cf.~\cite{Hof06}) and a different, but seemingly equivalent approach based on variational inequalities (cf.~\cite{Fle12}) have been developed subsequently.
In the case of $\ell^1$-regularization it has been shown that source conditions are directly related to sparsity, hence error estimates have been derived with constants depending on the sparsity level.

However, one can also consider solutions which are only merely sparse, i.e. few of the components being large and the majority being small and decaying fast to zero. Such a model is actually more realistic in many cases, e.g. when applying wavelets to audio signals or natural images. It is the basis of compression algorithms that most of the coefficients are very small and can be ignored, i.e. set to zero. In inversion methods, this type of a-priori information can be analyzed for $\ell^1$-regularization with two perspectives. The first is to assume that the relevant solution is indeed sparse, i.e. we are interested in a sparse approximation $\tilde x^\dag$ to $x^\dag$. In such a case on should clearly analyze a systematic error $A(\tilde x^\dag - x^\dag)$ in addition to the usual error, which is however not straightforward under general assumptions. The second approach, which we will adopt in this paper, is to really analyze the error in approximating $x^\dag$ using a natural separation into the (few) large and the remaining small entries.

The further goal of the paper is twofold. On the one hand, we are going to derive novel convergence rates for Tikhonov regularized
solutions of linear ill-posed operator equations, where the penalty functional is the $\ell^1$-norm.
We will prove convergence rates when the exact solution
is not sparse but in $\ell^1(\N)$. Moreover, we will formulate the specific manifestations of solution
smoothness in this case, also essentially based on the decay rate of solution components.
On the other hand, we give a first example for an application of the variational inequality approach
(see for details~\cite{Flemmingbuch12,Grasm11,HKPS07,HofMat12,HofYam10}) when
neither source conditions nor approximate source conditions in the Banach space setting
(cf.~\cite[Section~3.2.2]{SKHK12}) are available. The necessity of exploiting source
conditions in the course of constructing variational inequalities for obtaining convergence rates in Tikhonov regularization was up to now considered as a weakness of the approach.

The paper is organized as follows: In Section 2 we will fix the basic problem setup, notations, and assumptions, and then proceed to an overview of smoothness conditions for proving convergence rates of variational regularizations in Section 3. In Section 4 we verify that non-sparse signals indeed fail to satisfy source conditions and even approximate source conditions. In Section 5 we derive a new variational inequality and use it to prove convergence rates for approximately sparse solutions.

\section{Problem Setting and Assumptions}\label{s2}
\setcounter{equation}{0}
\setcounter{theorem}{0}

Let $\widetilde A \in \mathcal{L}(\widetilde X,Y)$ be an injective and bounded linear operator mapping between the infinite dimensional Banach spaces $\widetilde X$ and $Y$ with norms $\|\cdot\|_{\widetilde X}$
and $\|\cdot\|_Y$ as well as with topological duals $\widetilde X^*$ and $Y^*$, respectively,
where we assume that the \emph{range} of $\widetilde A$ is \emph{not closed}. This means that we have $\range(\widetilde A) \not= \overline {\range(\widetilde A)}^{\,Y}$,  which is equivalent to the fact that the inverse
$${\widetilde A}^{-1}: \range(\widetilde A) \subset Y \to \widetilde X$$ is an unbounded linear operator and that there is no constant $0<C<\infty$ such that
$$\|\widetilde x\|_{\widetilde X} \le C\,\|\widetilde A \widetilde x\|_Y \quad \mbox{for all} \quad \widetilde x \in X.$$ Moreover, we denote by $\{u_k\}_{k \in \N} \subset \widetilde X$ a \emph{bounded Schauder basis} in $\widetilde X$. This means that there is some $K>0$ such that
\begin{equation} \label{eq:c1}
\|u_k\|_{\widetilde X} \le K \quad \mbox{for all} \quad k \in \N
\end{equation}
 and any element $\widetilde x \in \widetilde X$ can be represented as an infinite series $\widetilde x= \sum \limits _{k=1}^\infty x_k u_k$ convergent in $\widetilde X$ with uniquely determined coefficients $x_k \in \R$ in the sense of $\lim \limits _{n \to \infty} \|\widetilde x - \sum \limits _{k=1}^n x_k u_k\|_{\widetilde X}=0$. In the sequel we always consider the coefficients $x_k,\;k=1,2,...,$ as components of an infinite real sequence $x:=(x_1,x_2,...)$ and following the
 setting in \cite{Grasm09}  we assume that $L: \1 \to \widetilde X$ is the \emph{synthesis operator} defined as
$$Lx:=\sum \limits _{k=1}^\infty x_k u_k \in \widetilde X \quad \mbox{for} \quad x=(x_1,x_2,...) \in \1.$$
Evidently, $L$ is a well-defined, injective and, as one simply verifies, also bounded linear operator, i.e. $L \in \mathcal{L}(\1,\widetilde X)$.

As usual $$\|x\|_{\ell^q(\N)}:=\left( \sum \limits_{k=1}^\infty |x_k|^q\right)^{1/q} $$ describe the norms in the Banach spaces $\ell^q(\N),\;1 \le q<\infty$, and $\|x\|_{\ell^\infty(\N)}:=\sup \limits_{k \in \N} |x_k|$  the norm in  $\ell^\infty(\N)$. The same norm  $\|x\|_{c_0}:=\sup \limits_{k \in \N} |x_k|$
is used for the Banach space $c_0$ of infinite sequences tending to zero. By $\ell^0(\N)$ we denote the set of \emph{sparse} sequences, where $x_k \not=0$ only occurs for a finite number of components.

In this paper, the focus is on elements $\widetilde x =Lx \in  \widetilde X$ which correspond to sequences $x \in \1$
and we choose the Schauder basis $\{u_k\}_{k \in \N}$ such that
\begin{equation} \label{eq:c2}
\overline{\range(L)}^{\widetilde X}=\widetilde X.
\end{equation}
When setting $$X:=\1, \qquad A:=  \widetilde A \circ L \in \mathcal{L}(\1,Y),$$
noting that $A$ is also injective since $\widetilde A$ and $L$ are,
we are searching for \emph{stable approximations} to elements  $x \in \1$ satisfying the linear operator equation
\begin{equation} \label{eq:opeq}
A x \,=\,y, \qquad x \in X ,\quad y \in Y,
\end{equation}
in an approximate manner, where instead of the exact right-hand side $y \in \range(A)$ only \emph{noisy data} $y^\delta \in Y$ satisfying the noise model
\begin{equation}\label{eq:noise}
\|y-y^\delta\|_Y \le \delta
\end{equation}
with noise level $\delta>0$ are given.

\begin{proposition} \label{prop:ill}
Under the given setting including the conditions  (\ref{eq:c1}) and (\ref{eq:c2}) the linear operator equation (\ref{eq:opeq}) is ill-posed, i.e., we have $\range(A) \not= \overline {\range(A)}^{\,Y}$.
\end{proposition}
\begin{proof}
By the continuity of $\widetilde A$ and by \eqref{eq:c2} $\range(A)$ is dense in $R(\widetilde{A})$. If $\range(A)$ would
be closed then we had $\range(A)=\range(\widetilde{A})$ and hence $\range(\widetilde{A})$
would be closed, too, which contradicts our assumptions.
\end{proof}

In the sequel, let $\langle v^*,v\rangle_{B^*\times B}$ denote the dual paring of an element $v$ from the Banach
space $B$ with an element $v^*$ from its dual space $B^*$.
Furthermore, we denote by \linebreak $e_k:=(0,0,...,0,1,0,...)$, with $1$ at the $k$-th position for $k=1,2,...$,  the elements of the normalized canonical Schauder basis in $\ell^q(\N),\;1 \le q<\infty$, which is also a Schauder basis in $c_0$. That means, we find $\lim \limits_{n \to \infty}\|x-\sum \limits_{k=1}^n x_k e_k\|_{c_0}=0$ for all $x=(x_1,x_2,...) \in c_0$ and $\lim \limits_{n \to \infty}\|x-\sum \limits_{k=1}^n x_k e_k\|_{\ell^q(\N)}=0$ for all $x=(x_1,x_2,...) \in  \ell^q(\N), \;1 \le q<\infty$, but not for $q=\infty$. Moreover, we suppose that the following standing assumptions hold:

\begin{assumption} \label{ass:basic}
\begin{itemize} \item[]
\item[(a)] The element $\xdag \in \1$ solves the operator equation (\ref{eq:opeq}).
\item[(b)] We have the limit condition $\lim\limits_{k \to \infty} \|Ae_k\|_Y = 0$.
\item[(c)] For all $k\in \N$ there exist $f_k \in Y^*, \;f_k \not=0,$ such that $e_k=A^*f_k$, i.e., we have \\
 $x_k=\langle e_k,x\rangle_{\3 \times \1}=\langle f_k,Ax\rangle_{Y^* \times Y}\,$ for all $x=(x_1,x_2,...) \in \1$.
\end{itemize}
\end{assumption}

\begin{remark} \label{rem:ass}{\rm
One simply sees that Assumption~\ref{ass:basic} is appropriate for the introduced model of the linear operator equation (\ref{eq:opeq}) with injective forward operator $A=\widetilde A~\circ~L$ and its ill-posedness verified by Proposition~\ref{prop:ill}. Firstly, by item (c) the operator \linebreak $A: \1 \to Y$ is injective. Namely, we have $x_k=\langle e_k,x\rangle_{\3 \times \1}=0$ for all $k \in \N$ whenever $x=(x_1,x_2,...) \in \1$ and $Ax=0$. This yields $x=0$ and hence the injectivity of $A$.
From the injectivity of $A$, however, we have that $\xdag$ from item (a) of Assumption~\ref{ass:basic} is the uniquely determined
solution of (\ref{eq:opeq}) for given $y \in \range(A)$.
Secondly, from item (b) of Assumption~\ref{ass:basic} it follows that there is no constant $0<C<\infty$ such that $1=\|e_k\|_{\1} \le C\,\|Ae_k\|_Y$ for all $k \in \N$
and hence we have that $\range(A)$ is a non-closed subset of the Banach space $Y$. Consequently, the linear operator equation (\ref{eq:opeq}) is ill-posed under  Assumption~\ref{ass:basic} and the inverse operator $A^{-1}: \range(A) \subset Y \to \1$, which exists due to the injectivity of $A$, is an unbounded linear operator. Hence, the stable
approximate solution of (\ref{eq:opeq}) based on noisy data $y^\delta \in Y$ satisfying (\ref{eq:noise}) requires some kind of regularization.} \hfill\fbox{}
\end{remark}

Note that by the closed range theorem the range $\range(A^*)$ is a non-closed subset of $\3$, but not dense in $\3$,
since it is a subset of $c_0$, as the following proposition indicates, and $c_0$ is not dense in $\3$ with respect to
the supremum norm.

\begin{proposition} \label{pro:c0}
Under Assumption~\ref{ass:basic} we have $\overline {\range(A^*)}^{\,\3}=c_0$.
\end{proposition}
\begin{proof}
First we show that $\range(A^*)\subseteq c_0$. For $w \in Y^*$ we obvioulsy have $A^*w \in \3$
and
$$|[A^*w]_k|=|\langle A^*w,e_k\rangle_{\3 \times \1}|=|\langle w,Ae_k\rangle_{Y^* \times Y}|\le \|w\|_{Y^*}\|Ae_k\|_Y.$$
Thus by Assumption~\ref{ass:basic} (b), $A^*w \in c_0$.

It remains to show that each $z \in c_0$ can be approximated by a sequence $\{A^*w_n\}_{n \in N}$
with $w_n \in Y^*$. For this purpose we define $w_n= \sum \limits_{k=1}^n z_k f_k$ with $f_k$ from Assumption~\ref{ass:basic} (c).
Then
$$\|z-A^*w_n\|_{\3}= \left\|\sum \limits_{k=n+1}^\infty z_ke_k\right\|_{\3} $$
and therefore $ \|z-A^*w_n\|_{\3}\to 0$ as $n \to \infty$, since $z \in c_0$.
\end{proof}

\begin{remark} \label{rem:additional} {\rm
An inspection of the proof shows that for Proposition~\ref{pro:c0} the condition (b) of Assumption~\ref{ass:basic} can be weakened to\\
$\hspace*{0.2cm} (b')\;\, \mbox{For all}\; k \in \N\; \mbox{we have weak convergence}\; Ae_k \rightharpoonup 0\;\mbox{in}\; Y.$

Further, item (c) of Assumption~\ref{ass:basic} implies that
\begin{equation} \label{eq:fk}
|x_k|\le \|f_k\|_{Y^*}\,\|Ax\|_Y \qquad \mbox{for all} \quad x \in \1 \quad \mbox{and} \quad k \in \N.
\end{equation}
Applying this inequality to $x=e_k$ we obtain $1\leq\|f_k\|_{Y^*}\,\|Ae_k\|_Y$ and therefore on the
one hand
\begin{equation}\label{eq:fk_lower}
\|f_k\|_{Y^*}\geq\frac{1}{\|A\|}
\end{equation}
and on the other hand by exploiting item (b)
\begin{equation}\label{eq:fk_infty}
\lim_{k\to\infty}\|f_k\|_{Y^*}\geq\lim_{k\to\infty}\frac{1}{\|Ae_k\|_Y}=\infty.
\end{equation}

If $\ell^\infty_g(\N)$ denotes the weighted $\ell^\infty$-space with positive weights $g=(1/\|f_1\|_{Y^*},1/\|f_2\|_{Y^*},...)$ and norm
$$\|x\|_{\ell^\infty_g(\N)}:= \sup \limits_{k \in \N}\,\frac{|x_k|}{\|f_k\|_{Y^*}},$$
from \eqref{eq:fk_lower} we have that
$$\|x\|_{\ell^\infty_g(\N)}\leq\|A\|\|x\|_{\3}  \qquad \mbox{for all} \quad x \in \3$$
and from (\ref{eq:fk}) that
$$\|x\|_{\ell^\infty_g(\N)}\le \|Ax\|_Y  \qquad \mbox{for all} \quad x \in \1$$
Hence the norm in $\ell^\infty_g(\N)$ is weaker than the standard supremum norm in $\3$ and we have $0<C<\infty$ such that $\|Gx\|_{\ell^\infty_g(\N)} \le C\,\|Ax\|_Y$ for all $x \in\1$, where $G$ is the embedding operator from $\1$ to $\ell^\infty_g(\N)$
the behaviour of which characterizes the nature of ill-posedness of problem (\ref{eq:opeq}), and we refer to \cite[Remark~3.5]{Lorenz08} for similar considerations in $\ell^q$-regularization. We also mention that an assumption similar to (c) also appears in \cite[Assumption~4.1(a)]{BreLor09}.} \hfill\fbox{}
\end{remark}

\begin{example}[diagonal operator] \label{ex:diag}
{\rm In order to get some more insight into details, we consider for a separable infinite dimensional  \emph{Hilbert space} $\widetilde X$  and a selected orthonormal basis $\{u_k\}_{k \in \N}$ in $\widetilde X$ the compact linear operator
$\widetilde A: \widetilde X \to \widetilde X$ with \emph{diagonal structure}. That means we have $Y=\widetilde X$ and
$\widetilde A$ possesses the singular system $\{\sigma_k,u_k,u_k\}_{k \in \N}$
such that for the decreasingly ordered sequence of positive singular values $\{\sigma_k\}_{k \in \N}$, tending to zero as $k \to \infty$, the equations $\widetilde Au_k=\sigma_k u_k$ and $\widetilde A^*u_k=\sigma_k u_k$ are valid for all $k \in N$. For
$\widetilde x= \sum \limits_{k \in N} x_k u_k$ we have the inner products $\langle \widetilde x, u_k \rangle_{\widetilde X}$ as square-summable components $x_k$ in  the infinite sequence $x=(x_1,x_2,...)$. Then the bounded linear synthesis operator \linebreak $L: X=\1 \to \widetilde X$ is the composition $L=\mathcal{U} \circ \mathcal{E}$ of the injective embedding operator  $\mathcal E$ from $\1$ to $\2$ with $\overline{\range(\mathcal E)}^{\2}=\2$ and the unitary operator $\mathcal U$, which characterizes
the isometry between $\2$ and $\widetilde X$. Hence, $\overline{\range(L)}^{\widetilde X}=\widetilde X$ and both conditions (\ref{eq:c1}) and (\ref{eq:c2}) are satisfied for that example.

 The injective linear operator $A=\widetilde A \circ L: X \to Y$ is as a composition of a bounded and a compact linear operator also compact
 and item (b) of Assumption~\ref{ass:basic} is satisfied because
of $\|Ae_k\|_{Y}= \|Ae_k\|_{\widetilde X}=\|\widetilde Au_k\|_{\widetilde X}=\sigma_k \to 0$ as $k \to \infty$. Since we have $[A^*u_k]_k=\sigma_k$ and $[A^*u_k]_m=0$ for $m \not= k$ as a consequence of
$\widetilde A^*u_k=\sigma_k u_k, \;k \in \N,$ item (c) is fulfilled with
$f_k=\frac{1}{\sigma_k} u_k,\;k \in \N,$ and $\|f_k\|_{Y^*}=\|f_k\|_{\widetilde X}=\frac{1}{\sigma_k}$
tends to infinity as $k \to \infty$.} \hfill\fbox{}
\end{example}

Our focus in on situations where we conjecture that the solutions of (\ref{eq:opeq}) are sparse, i.e. $x \in \ell^0(\N)$, or at least that the coefficients $x_k$ in $x \in \1$ are negligible for sufficiently large $k \in \N$.
Then it makes sense to use the \emph{$\ell^1$-regularization}, and the regularized solutions $\xad \in \1$ are minimizers
of the extremal problem
\begin{equation} \label{eq:Tik}
T_\alpha(x):=\frac{1}{p}\|Ax-y^\delta\|^p_Y+\alpha\,\|x\|_{\1} \to \min, \quad \mbox{subject to}\quad x \in X=\1,
\end{equation}
of Tikhonov type with regularization parameters $\alpha>0$ and a misfit norm exponent $1<p<\infty$.
Then the sublinear and continuous penalty functional $\Omega(x):=\|x\|_{\1}$ possessing finite values for all $x \in X$ is convex and lower semi-continuous. Since $X=\1$ is a \emph{non-reflexive} Banach space we need some topology $\tau_X$ in $X$ which is weaker than the canonical weak topology in $X$ in order
to ensure stabilizing properties of the regularized solutions. In other words, $\Omega$ must be a \emph{stabilizing functional} in the sense that the sublevel sets $$\mathcal{M}^\Omega(c):=\{x \in X:\,\Omega(x) \le c\}$$ are $\tau_X$-sequentially precompact subsets of $X$ for all $c \ge 0$. Since $Z:=c_0$ with $Z^*=\1$ is a separable predual Banach space of $X$, we can use the
associated weak$^*$-topology as $\tau_X$ (cf.~\cite[Remark~4.9 and Lemma~4.10]{SKHK12}). Note that $\Omega$ under
consideration here is also sequentially lower semi-continuous with respect to the weak$^*$-topology.
If the operator $A$ can be proven to transform weak$^*$-convergent sequences in $\1$ to weakly convergent sequences in the Banach space $Y$,
which will be done by Lemma~\ref{lem:star} below, then existence and stability of regularized solutions can be ensured. We refer for details to \cite[\S3]{HKPS07}  and also to \cite{BurResHe,Grasm09}.
Precisely, there even exist uniquely determined minimizers $\xad$ for all $\alpha>0$ and $y^\delta \in Y$, because the Tikhonov functional $T_\alpha$ is strictly convex due to the injectivity of $A$. Moreover, the regularized solutions $\xad$ are stable with respect to small data changes, and we have  $\xad \in  \ell^0(\N)$ for fixed $\alpha>0$.
The last fact is proven in \cite[Lemma~2.1]{Lorenz08} if $Y$ is a Hilbert space.
For Banach spaces $Y$ the proof remains the same if one observes that for each minimizer
$\xad$ there is some $\xi\in\3$ such that
$$\xi\in A^\ast\partial\left(\tfrac{1}{p}\|\cdot-y^\delta\|_Y^p\right)(A\xad)
\subseteq\range(A^\ast)\subseteq c_0
\quad\text{and}\quad
-\xi\in\partial(\alpha\|\cdot\|_{\1})(\xad).$$

\begin{lemma} \label{lem:star}
Under the assumptions stated above the operator $A: \1 \to Y$ is always
sequentially weak*-to-weak continuous, i.e., for a weakly convergent sequence $x^{(n)} \rightharpoonup^*\,
\overline x$ in $\ell^1$ we have weak convergence as $Ax^{(n)} \rightharpoonup A \overline x$ in $Y$.
\end{lemma}
\begin{proof}
Since the separable Banach space $c_0$, which has the same supremum norm like $\ell^\infty$, is a predual space of $\ell^1$, i.e.~any element $x
=(x_1,x_2,...) \in \ell^1$ is a bounded linear functional on $c_0$, the weak$^*$-convergence $x^{(n)} \rightharpoonup^*\,
\overline x$ in $\ell^1$ can be written as
$$\lim \limits _{n \to \infty}\langle x^{(n)},g\rangle_{\ell^1\times c_0}= \lim \limits _{n \to \infty}\sum \limits_{k \in \N} x^{(n)}_k g_k = \sum \limits_{k \in \N}  \overline x_k g_k = \langle \overline x,g\rangle_{\ell^1\times c_0} \quad \mbox{for all}\quad g=(g_1,g_2,...) \in c_0.$$
With the bounded linear operator $A^*: Y^* \to \ell^\infty$  we can further conclude from $\mathcal{R}(A^*) \subseteq c_0$, which follows from Proposition~\ref{pro:c0} , that $A^*f \in c_0$ for all $f \in Y^*$ and that
$$\langle f,Ax^{(n)}\rangle_{Y^* \times Y}=\langle A^*f,x^{(n)}\rangle_{\ell^\infty,\ell^1}= \langle x^{(n)},A^*f\rangle_{\ell^1 \times c_0}\quad \mbox{for all}\quad f \in Y^*.$$
Hence we have
$$\lim \limits _{n \to \infty}\langle f,Ax^{(n)}\rangle_{Y^*Y}=\lim \limits _{n \to \infty}
\langle x^{(n)},A^*f\rangle_{\ell^1 \times c_0}=\langle \overline x,A^*f\rangle_{\ell^1\times c_0}=\langle f, A\overline x\rangle _{Y^* \times Y}
\quad \mbox{for all}\quad f \in Y^*,$$
which yields the weak convergence in $Y$ and completes the proof.

\end{proof}

\section{Manifestations of smoothness for convergence rates} \label{s3}
\setcounter{equation}{0}
\setcounter{theorem}{0}

It is well-known that for linear ill-posed operator equations (\ref{eq:opeq}) with $A \in \mathcal{L}(X,Y),$ formulated in Banach spaces $X$ and $Y$ with some solution $\xdag \in X$, \emph{convergence rates of regularized solutions}
\begin{equation} \label{eq:Genrates}
E(\xad,\xdag)=\mathcal{O} (\varphi(\delta)) \qquad \mbox{as} \qquad\delta \to 0
\end{equation}
evaluated by some \emph{nonnegative error measure} $E$ and
some \emph{index function} $\varphi$ require additional properties of $\xdag$ which express some kind of smoothness of the solution with respect to the forward operator $A:X \to Y$ and its adjoint $A^*:Y^* \to X^*$. We call a function
$\varphi: (0,\infty) \to (0,\infty)$ index function if it is continuous, strictly increasing and satisfies the
 limit condition $\lim \limits_{t \to +0} \varphi(t)=0$. Moreover, we denote by
$\xad$ the minimizers of
$$\frac{1}{p}\|Ax-y^\delta\|^p_Y+\alpha\,\Omega(x) \to \min, \quad \mbox{subject to}\quad x \in X,$$
for $1<p <\infty$ and some convex stabilizing penalty functional $\Omega: X \to [0,\infty]$. Then the original form of smoothness
is given by \emph{source conditions}
\begin{equation} \label{eq:benchmark}
\xi=A^*w,\qquad w \in Y^*,
\end{equation}
for subgradients $\xi \in \partial \Omega(\xdag) \subset X^*$, and the error can be measured by the Bregman distance
\begin{equation} \label{eq:Bregman}
E(\xad,\xdag):=\Omega(x)-\Omega(\xdag)-\langle \xi,x-\xdag \rangle _{X^* \times X}
\end{equation}
as introduced in \cite{BurOsh04}. Then convergence rates (\ref{eq:Genrates}) with $\varphi(t)=t$ can be derived under appropriate choices of the regularization parameter $\alpha>0$ (cf.~\cite{BurOsh04,HKPS07,Scherzetal09}).

If the subgradient $\xi \in X^*$ fails to satisfy (\ref{eq:benchmark}), then one can use \emph{approximate source conditions} and ask for the degree of violation of $\xi$ with respect to the benchmark source condition (\ref{eq:benchmark}).  This violation is, for example, expressed by properties of the strictly positive, convex and nonincreasing distance function
\begin{equation} \label{eq:distfct}
d_\xi(R):=\inf \limits _{w \in Y^*: \,\|w\|_{Y^*} \le R}\|\xi-A^*w\|_{X^*}.
\end{equation}
If the limit condition
\begin{equation} \label{eq:limit}
\lim \limits _{R \to \infty} d_\xi(R)=0
\end{equation}
holds, then one can prove convergence rates (\ref{eq:Genrates}) with $E$ from (\ref{eq:Bregman}) and $\varphi$ depending on $d_\xi$  (cf.~\cite{BoHo10,HeinHof09}, \cite[Chapter~3]{SKHK12}, and \cite[Appendix A]{HofMat12}). If, for example, the Bregman
distance is $q$-coercive with $q \ge 2$ and $1/q+1/q^*=1$, then we have
$$\varphi(t)=\left[d_\xi \left(\Phi^{-1}(t)\right)\right]^{q*},\qquad \mbox{where} \quad \Phi(R):= \frac{[d_\xi(R)]^{q^*}}{R}.$$
If, however, the distance function $d_\xi$ does not satisfy (\ref{eq:limit}), this approach fails. As mentioned in
\cite{BoHo10} such situation is only possible if the biadjoint operator $A^{**}: X^{**} \to Y^{**}$ mapping between
the bidual spaces of $X$ and $Y$ is not injective.

An alternative manifestation of solution smoothness is given by \emph{variational inequalities} (cf., e.g., \cite{HKPS07,HofYam10}), where in the sequel our focus will be on the variant formulated in Assumption~\ref{ass:vi}, originally suggested in \cite{Flemmingbuch12,Grasm10}.

\begin{assumption} \label{ass:vi}
For given nonegative error measure $E$, convex stabilizing functional $\Omega$ and $\xdag \in X$, let there exist a concave
index function $\varphi$, a set $\M \subseteq X$ with $\xdag \in \M$ and constants $\beta>0$ as well as
$C>0$ such that
\begin{equation}\label{eq:vi}
\beta\,E(x,\xdag) \le \Omega(x)-\Omega(\xdag)+ C\,\varphi \left(\|A(x-\xdag)\|_Y\right)  \qquad \mbox{for all} \qquad x \in \M.
\end{equation}
\end{assumption}

In the case of an \emph{appropriate choice of the regularization parameter} $\alpha>0$ the index function $\varphi$
in the variational inequality (\ref{eq:vi}) immediately determines the convergence rate of Tikhonov-regularized solutions $\xad$ to $\xdag$. Proofs of the assertions of the following proposition can be found in \cite[Theorem~1]{HofMat12} and \cite[Chapter~4]{Flemmingbuch12}. We also refer to \cite{AR2}.

\begin{proposition} \label{prop:basic}
Let the regularization parameter be chosen a priori as $\alpha=\alpha(\delta):=\frac{\delta^p}{\varphi(\delta)}$ or a posteriori
as $\alpha=\alpha(\delta,y^ \delta)$ according to the strong discrepancy principle
\begin{equation} \label{eq:strdis}
\tau_1 \,\delta \le \|A \xad-y^\delta\|_Y \le \tau_2 \,\delta
\end{equation}
for prescribed constants $1 \le \tau_1 \le \tau_2<\infty$.  Then under Assumption~\ref{ass:vi} we have the convergence rate (\ref{eq:Genrates})
whenever in both cases all regularized solutions
$\xad$ for sufficiently small $\delta>0$ belong to $\M$.
\end{proposition}

\begin{remark} \label{rem:dis} {\rm
As was shown in \cite{HofMat12} the strong discrepancy principle (\ref{eq:strdis}) in Proposition~\ref{prop:basic}  can be replaced with the more traditional (sequential) discrepancy principle, where with sufficiently large $\alpha_0>0,\;0<\zeta<1$,  and $\Delta_\zeta:=\{\alpha_j:\,\alpha_j:=\zeta^j\alpha_0,\;j=1,2,... \}$ the regularization parameter $\alpha=\alpha(\delta,y^ \delta)$ is chosen as the largest parameter within $\Delta_\zeta$ such that $\|Ax_\alpha-y^\delta\|_Y \le \tau\,\delta $ for prescribed $\tau>1$.
This, however, is more of interest if the forward operator is nonlinear and duality gaps can hinder
regularization parameters $\alpha=\alpha(\delta,y^ \delta)$ to fulfil (\ref{eq:strdis}) simultaneously with
lower and upper bounds.}  \hfill\fbox{}
\end{remark}

\section{Failure of approximate source conditions} \label{s4}
\setcounter{equation}{0}
\setcounter{theorem}{0}

Now we come back to the situation $X=\1$ of $\ell^1$-regularization introduced in Section~\ref{s2} and pose the following additional assumption.

\begin{assumption} \label{ass:nonsparse}
Let the solution $\xdag=(\xdag_1,\xdag_2,...) \in \1$ of (\ref{eq:opeq}) be non-sparse, i.e., $\xdag \notin \ell^0(\N)$
and hence there is an infinite subsequence
$\{\xdag_{k_n}\not=0\}_{n=1}^\infty$ of nonzero components of $\xdag$.
\end{assumption}

\begin{lemma} \label{lem:jens}
Let Assumptions \ref{ass:basic} and \ref{ass:nonsparse} hold and
let $\xi\in \3$ be an arbitrary element of $X^*=\3$.
Then on the one hand, $d_\xi(R)\to 0$ for $R\to\infty$ if and only if $\xi\in c_0$. On the other hand,
if $s:=\limsup \limits_{k\to\infty}\vert\xi_k\vert>0$ we have $d_\xi(R)\geq s$ for
all $R\geq 0$.
\end{lemma}
\begin{proof}
By definition of the distance function, $d_\xi(R)\to 0$ for $R\to\infty$ if and only if $\xi\in\overline{\range(A^\ast)}^{\,\3}$.
Hence the first part of the assertion is a consequence of Proposition~\ref{pro:c0} where $\overline{\range(A^\ast)}^{\,\3}=c_0$ was proven.
\par
For proving the second part of the assertion we take a subsequence $\{\xi_{l_n}\}_{n\in\N}$
with \linebreak $|\xi_{l_n}|_{\3}\to s$ as $n\to\infty$ and assume $w\in Y^\ast$ with
$\|w\|_{Y^*}\leq R$. Because of
$$\vert[A^\ast w]_{l_n}\vert\leq\|w\|_{Y^\ast}\|Ae_{l_n}\|_Y\leq R\|Ae_{l_n}\|_Y\to 0 \quad \mbox{as} \quad n \to \infty,$$
we see that $\|\xi-A^\ast w\|_{\3}\geq\sup_{n\in\N}\vert\xi_{l_n}-[A^\ast w]_{l_n}\vert\geq s$.
Thus, $d_\xi(R)\geq s$  for all $R>0$. This completes the proof.
\end{proof}

\begin{proposition} \label{prop:failure}
Under the Assumptions~\ref{ass:basic} and \ref{ass:nonsparse} the benchmark source condition (\ref{eq:benchmark}) always fails. Also the method of approximate source conditions is not applicable, because we have for the corresponding distance function
\begin{equation} \label{eq:splimit}
d_\xi(R)=\inf \limits _{w \in Y^*: \,\|w\|_{Y^*} \le R}\|\xi-A^*w\|_{\3}\,\ge\,1 \qquad \mbox{for all} \qquad R>0.
\end{equation}
\end{proposition}

\begin{proof}
As is well-known the subgradients $\xi =(\xi_1,\xi_2,...) \in \partial \|x\|_{\1}$ can be made explicit as
$$ \xi_k \, = \,  \begin{cases} \quad 1 \quad \; \qquad \quad \mbox{if} \qquad x_k>0,\\ \in [-1,1] \qquad \mbox{if} \qquad x_k=0, \\\;\; -1 \qquad \quad \quad \mbox{if} \qquad x_k<0, \end{cases} \qquad k \in \N. $$
So we have $\xi \notin c_0$ by Assumption~\ref{ass:nonsparse}. Moreover, by this assumption we also have $|\xi_{k_n}|=1$ for all $n \in \N$, that is $\limsup \limits
_{k \to \infty}|\xi_k|=1$. Lemma~\ref{lem:jens} thus shows $ d_\xi(R)\ge 1$ for all $R > 0$. This means that
$d_\xi(R)$ does not tend to zero as $R \to \infty$ and, in particular, that $\xi \notin \range(A^*)$ as a direct
consequence of $\xi \notin c_0$, since this would imply $d_\xi(R)=0$ for sufficiently large $R>0$.
\end{proof}

\begin{remark} \label{rem:doblestar} {\rm
Since (\ref{eq:splimit}) implies that
$d_\xi(R) \not \to 0$ as $R \to \infty$ and (\ref{eq:limit}) fails, this is an example for the case that the biadjoint operator $A^{**}$ is not injective although $A$ is injective, where we have $A^{**}: (\3)^* \to Y^{**}$ here. } \hfill\fbox{}
\end{remark}

\section{Derivation of a variational inequality} \label{s5}
\setcounter{equation}{0}
\setcounter{theorem}{0}

As outlined in the previous section source conditions and even approximate
source conditions do not hold if the searched for solution $x^\dagger \in \1$ to equation (\ref{eq:opeq})
is not sparse. In the following we derive a variational inequality as
introduced in Assumption~\ref{ass:vi} with $\M=X=\1$ and $\beta=1$ in the case of a non-sparse solution.
By Proposition~\ref{prop:basic} we then directly obtain a convergence rate. Since the index function $\varphi$
constructed below is not explicitly available for choosing $\alpha>0$, a posteriori choices of the regularization
parameter are to be preferred, and in particular the strong discrepancy principle (\ref{eq:strdis}) ensures
the convergence rate
$$ \|\xad-\xdag\|_{\1}=\mathcal{O} (\varphi(\delta)) \qquad \mbox{as} \qquad\delta \to 0. $$

We need the following lemma.

\begin{lemma}\label{lem:sums}
For all $x\in\1$ and all $n\in\N$ the inequality
$$\|x-\xdag\|_{\1}-\|x\|_{\1}+\|\xdag\|_{\1}
\leq 2\left(\sum_{k=n+1}^\infty\vert\xdag_k\vert+\sum_{k=1}^n\vert x_k-\xdag_k\vert\right)$$
is true.
\end{lemma}

\begin{proof}
For $n\in\N$ define the projection $P_n:\1\rightarrow\1$ by $P_nx:=(x_1,\ldots,x_n,0,\ldots)$.
Then obviously
$$\|x\|_{\1}=\|P_nx\|_{\1}+\|(I-P_n)x\|_{\1}$$
for all $x\in\1$.
Based on this equality we see that
\begin{align*}
\lefteqn{\|x-\xdag\|_{\1}-\|x\|_{\1}+\|\xdag\|_{\1}}\\
&\qquad=\|P_n(x-\xdag)\|_{\1}+\|(I-P_n)\xdag\|_{\1}+\|(I-P_n)(x-\xdag)\|_{\1}\\
&\qquad\qquad-\|(I-P_n)x\|_{\1}+\|P_n\xdag\|_{\1}-\|P_n x\|_{\1}.
\end{align*}
Consequently, together with the triangle inequalities
$$\|(I-P_n)(x-\xdag)\|_{\1}\leq\|(I-P_n)x\|_{\1}+\|(I-P_n)\xdag\|_{\1}$$
and
$$\|P_n\xdag\|_{\1}\leq\|P_n(x-\xdag)\|_{\1}+\|P_n x\|_{\1}$$
we obtain
$$\|x-\xdag\|_{\1}-\|x\|_{\1}+\|\xdag\|_{\1}
\leq2\left(\|P_n(x-\xdag)\|_{\1}+\|(I-P_n)\xdag\|_{\1}\right).$$
\end{proof}

\begin{theorem} \label{thm:main}
The variational inequality \eqref{eq:vi} holds true with $\beta=1$, $\M=X=\1$,
$E(x,x^\dagger)=\|x-x^\dagger\|_{\1}$, $\Omega(x)=\|x\|_{\1}$, and
\begin{equation} \label{eq:main}
\varphi(t)=2\inf_{n\in\mathbb{N}}\left(\sum_{k=n+1}^\infty\vert x^\dagger_k\vert
+t\sum_{k=1}^n\|f_k\|_{Y^*}\right).
\end{equation}
The function $\varphi$ is a concave index function.
\end{theorem}

\begin{proof}
By Assumption~\ref{ass:basic}(c) we have
$$\sum_{k=1}^n\vert x_k-x^\dagger_k\vert
=\sum_{k=1}^n\langle e_k,x-x^\dagger\rangle_{\3\times\1}
\leq\sum_{k=1}^n\|f_k\|_{Y^\ast}\|A(x-x^\dagger)\|_Y$$
for all $n\in\N$ and with the help of Lemma~\ref{lem:sums} we obtain
$$\|x-\xdag\|_{\1}-\|x\|_{\1}+\|\xdag\|_{\1}
\leq 2\left(\sum_{k=n+1}^\infty\vert\xdag_k\vert+\|A(x-x^\dagger)\|_Y\sum_{k=1}^n\|f_k\|_{Y^\ast}\right)$$
for all $n\in\N$.
Taking the infimum over $n$ on the right-hand side yields the desired
variational inequality.
\par
It remains to show that $\varphi$ is a \emph{concave index function}.
As an infimum of affine functions
$\varphi$ is concave and upper semi-continuous. Since $\varphi(t)<\infty$ for $t\in[0,\infty)$,
$\varphi$ is even continuous on $(0,\infty)$. Together with $\varphi(0)=0$
the upper semi-continuity in $t=0$ yields $\varphi(t)\to 0$ if $t\to+0$,
that is, continuity in $t=0$.
To show that $\varphi$ is strictly increasing take $t_1,t_2\in[0,\infty)$
with $t_1<t_2$. Since by \eqref{eq:fk_infty} the infimum in the definition
of $\varphi(t_2)$ is attained at some $n_2\in\N$, we obtain
$$\varphi(t_1)\leq 2\left(\sum_{k=n_2+1}^\infty\vert x^\dagger_k\vert
+t_1\sum_{k=1}^{n_2}\|f_k\|_{Y^*}\right)
<\varphi(t_2).$$
\end{proof}

\begin{example}[H\"older rates] \label{ex:power}
{\rm As one sees the rate function $\varphi$ in Theorem~\ref{thm:main} depends on decay properties of the solution components $|\xdag_k|$ for $k \to \infty$ and on the growth properties of the finite sums $\sum \limits_{k=1}^n \|f_k\|_{Y^*}$ for $n \to \infty$. If decay and growth are of monomial type as
\begin{equation} \label{eq:power}
\sum \limits_{k=n+1}^\infty |\xdag_k|\le K_1\, n^{-\mu},  \qquad \sum \limits_{k=1}^n \|f_k\|_{Y^*} \le K_2\, n^\nu, \qquad \mu,\nu>0,
\end{equation}
with some constants $K_1,K_2>0$,
then we have from Theorem~\ref{thm:main} (see formula (\ref{eq:main})) and Proposition~~\ref{prop:basic} that the strong discrepancy principle for choosing the regularization parameter $\alpha>0$ yields the \emph{H\"older convergence rates}
\begin{equation} \label{eq:hoelder}
\|\xad-\xdag\|_{\1} = \mathcal{O} \left(\delta^\frac{\mu}{\mu+\nu}\right) \qquad \mbox{as} \qquad\delta \to 0.
\end{equation}
Whenever exponents $\hat \mu>1,\,\hat \nu>0,$ and constants $\hat K_1,\hat K_2>0$ exist such that
$$ |\xdag_k|\le \hat K_1\, n^{-\hat \mu},  \qquad \|f_k\|_{Y^*} \le \hat K_2\, n^{\hat\nu}, \qquad \mbox{for all} \quad k \in \N,$$
the rate result (\ref{eq:hoelder}) can be rewritten as
$$\|\xad-\xdag\|_{\1} = \mathcal{O} \left(\delta^\frac{\hat\mu-1}{\hat\mu+ \hat\nu}\right) \qquad \mbox{as} \qquad\delta \to 0.$$
In particular, for the compact operator $A$ in the diagonal case of Example~\ref{ex:diag}, the exponent $\hat \nu>0$ can be interpreted as the degree of ill-posedness expressed by the operator $A$ if the decay rate
 of the singular values of $A$ is of the form $\sigma_k \sim k^{-\hat \nu}$.}
\end{example}

\begin{remark} {\rm
If $x^\dagger$ is sparse and if $\bar{n}\in\N$ is the largest $n$ for which
$x^\dagger_n\neq 0$, then the theorem yields a variational inequality with
$$\varphi(t)\leq 2\left(\sum_{k=1}^{\bar{n}}\|f_k\|_{Y^*}\right)t.$$
Consequently the $\mathcal{O}$-constant in the corresponding convergence
rate $\|x_\alpha^\delta-x^\dagger\|_{\1}=\mathcal{O}(\delta)$
depends on the size of the support of $x^\dagger$. }\hfill\fbox{}
\end{remark}

Theorem~\ref{thm:main} yields a function $\varphi$ and a corresponding
variational inequality \eqref{eq:vi} for arbitrary $x^\dagger\in X=\1$.
But as a consequence of Assumption~\ref{ass:basic} (c) the proof of Theorem~\ref{thm:main}
only works for injective operators $A$. The converse assertion would be that
if for each $\bar{x}\in X$ there is some index function $\varphi$ such that
the variational inequality
$$\|x-\bar{x}\|_{\1}\leq\|x\|_{\1}-\|\bar{x}\|_{\1}+\varphi\bigl(\|A(x-\bar{x})\|_Y\bigr)\quad\text{for all}\quad x\in X$$
is true, then $A$ is injective.
The following proposition states that even the existence of only one such $\bar{x}$
is sufficient for injectivity of $A$.

\begin{proposition}
If there is some $\bar{x}\in \1$ with $\bar{x}_k\neq 0$ for all $k\in\N$ such that
the variational inequality
\begin{equation}\label{eq:vi_tilde}
\beta\|x-\bar{x}\|_{\1}\leq\|x\|_{\1}-\|\bar{x}\|_{\1}+\varphi\bigl(\|A(x-\bar{x})\|_Y\bigr)\quad\text{for all}\quad x\in \1
\end{equation}
is satisfied with $\beta>0$ and an arbitrary index function $\varphi$, then $A$ is injective.
\end{proposition}

\begin{proof}
For $x\in\nullspace(A)\setminus\{0\}$, where $\nullspace(A)$ denotes the null-space of $A$, the
variational inequality (\ref{eq:vi_tilde}) applied to $\bar{x}+x$ yields
$$h(x):=\beta\|x\|_{\1}-\|\bar{x}+x\|_{\1}+\|\bar{x}\|_{\1}
\leq\varphi(\|Ax\|_Y)=0.$$
If there is some $t>0$ such that $h(tx)>0$ we have a contradiction to
$tx\in\nullspace(A)$. Thus, we can assume $h(tx)\leq 0$ for all $t>0$.
\par
Define two index sets
\begin{align*}
I_1(t)&:=\{k\in\N: t\vert x_k\vert\leq\vert \bar{x}_k\vert\},\\
I_2(t)&:=\{k\in\N: t\vert x_k\vert>\vert \bar{x}_k\vert\},
\end{align*}
where $t>0$. Simple calculations then show that
$$\vert \bar{x}_k-tx_k\vert+\vert \bar{x}_k+tx_k\vert
=\begin{cases}2\vert \bar{x}_k\vert,&k\in I_1(t),\\
2t\vert x_k\vert,&k\in I_2(t),\end{cases}$$
for all $k\in\N$. Thus we have for all $t>0$  
\begin{align*}
h(-tx)
&=2\beta\|tx\|_{\1}-\|\bar{x}-tx\|_{\1}-\|\bar{x}+tx\|_{\1}
+2\|\bar{x}\|_{\1}-h(tx)\\
&\geq 2\beta\|tx\|_{\1}-\|\bar{x}-tx\|_{\1}-\|\bar{x}+tx\|_{\1}
+2\|\bar{x}\|_{\1}\\
&=2\left(\|\bar{x}\|_{\1}-\sum_{k\in I_1(t)}\vert \bar{x}_k\vert\right)
+2t\left(\beta\|x\|_{\1}-\sum_{k\in I_2(t)}\vert x_k\vert\right).
\end{align*}
\par
Now choose $n\in\N$ such that there is some $k_0\in\{1,\ldots,n\}$
with $x_{k_0}\neq 0$ and such that
$$\sum \limits_{k=n+1}^\infty\vert x_k\vert<\beta\|x\|_{\1}.$$
Set
$$\bar{t}:=\min\left\{\frac{\vert \bar{x}_k\vert}{\vert x_k\vert}:k\in\{1,\ldots,n\},\;x_k\neq 0\right\}.$$
Then
$$\sum_{k\in I_2(\bar{t})}\vert x_k\vert\leq\sum_{k=n+1}^\infty\vert x_k\vert<\beta\|x\|_{\1}$$
and therefore $h(-\bar{t}x)>0$, which contradicts $-\bar{t}x\in\nullspace(A)$.
\par
In all cases the assumption $x\in\nullspace(A)\setminus\{0\}$ led to a
contradiction. Thus, $\nullspace(A)=\{0\}$ is proven.
\end{proof}

\section{Conclusion}

We have shown that, in the case of $\ell^1$-regularization, variational inqualities can significantly increase the range of solutions for which convergence rates can be shown compared to source conditions. Of course, the results are still preliminary since they rely on the injectivity of the operator $A$, which is an unwanted feature in many setups, in particular if motivated by compressed sensing. However, it provides an interesting insight into the current borders of regularization theory and a strong motivation to study variational inequality approaches in particular for singular regularization functionals.

Thinking about potential extensions of the approach and weakening of the assumptions one observes that currently several steps are based on the choice of "minimal" subgradients, i.e. the entries of the subgradient are set to zero outside the support of the solution.
From a source condition perspective, it can be seen as the assumption that for all one-sparse solutions a very strong source condition is satisfied by the minimal subgradient. A feasible approach to extend the results of this paper might be to consider larger classes of subgradients, whose absolute value should however be bounded away from one or decay to zero outside the support. The exact type of needed condition remains to be determined in the future.

\section*{Acknowledgments}

This work was supported by the German Science Foundation (DFG) under
grants \linebreak BU~2327/6-1 (M.~Burger) and HO~1454/8-1 (B.~Hofmann).
We express our thanks to Peter Math\'e (Berlin) for important hints concerning projections and their properties in $\1$, which allowed us to formulate and prove Lemma~\ref{lem:sums}, and to Radu Ioan Bo\c{t} (Chemnitz) for detailed remarks that essentially helped to improve the paper.



\end{document}